\documentclass[11pt,fleqn]{article}
\usepackage{amssymb}
\usepackage{amsfonts}
\usepackage{amsmath}
\usepackage{amsthm}
\usepackage{graphicx}
\usepackage{epsfig}
\usepackage{psfrag}
\usepackage{color}
\usepackage{dsfont}
\usepackage{mathtools}
\makeatletter
\xdef\@endgadget#1{{\unskip\nobreak\hfil\penalty50\hskip1em\hbox{}\nobreak
    \hfil#1\parfillskip=0pt\finalhyphendemerits=0\par}}
\def\@qedsymbol{${}_\blacksquare$}
\def\qed{\@endgadget{\@qedsymbol}}
\newtheorem{lemma}{Lemma}[section]
\newtheorem{theorem}[lemma]{Theorem}
\newtheorem{corollary}[lemma]{Corollary}
\newtheorem{example}[lemma]{Example}
\newtheorem{definition}[lemma]{Definition}

\newtheorem{proposition}[lemma]{Proposition}
\newtheorem{remark}[lemma]{Remark}

\newcommand{\Exp}{\mathrm{Exp\,}}
\newcommand{\Ln}{\mathrm{Ln\,}}
\newcommand{\diag}{\mathrm{diag\,}}
\newcommand{\rank}{\mathrm{rank\,}}
\newcommand{\im}{\mathrm{im\,}}
\newcommand{\spa}{\mathrm{span\,}}

\newcommand{\mR}{\mathbb{R}}

\newcommand{\red}{\color{black}}

\def\BibTeX{{\rm B\kern-.05em{\sc i\kern-.025em b}\kern-.08em
    T\kern-.1667em\lower.7ex\hbox{E}\kern-.125emX}}
    
\title{Complex and detailed balancing of chemical reaction networks revisited}
\author{A.J. van der Schaft\thanks{Arjan van der Schaft is with the Johann Bernoulli Institute for Mathematics and Computer
Science, University of Groningen, PO Box 407, 9700 AK, the Netherlands, +31-50-3633731, 
{\tt\small A.J.van.der.Schaft@rug.nl}}, S. Rao\thanks{Shodhan Rao is with Ghent University Global Campus,
119 Songdomunhwa-ro, Yeonsu-gu,
Incheon, South Korea 406-840, {\tt\small shodhan.rao@ghent.ac.kr}}, B. Jayawardhana
\thanks{Bayu Jayawardhana is with the Engineering and Technology Institute of Groningen, University of Groningen, the
Netherlands, +31-50-3637156, {\tt\small b.jayawardhana@rug.nl}}
}

\begin{document}
\maketitle
\begin{abstract} 
The characterization of the notions of complex and detailed balancing for mass action kinetics chemical reaction networks is revisited from the perspective of algebraic graph theory, in particular Kirchhoff's Matrix Tree theorem for directed weighted graphs. This yields an elucidation of previously obtained results, in particular with respect to the Wegscheider conditions, and a new necessary and sufficient condition for complex balancing, which can be verified constructively.
\end{abstract}
\section{Introduction} \label{sec:intro}
The notion of complex balancing of mass action kinetics chemical reaction networks, generalizing the classical notion of detailed balancing, dates back at least to the origin of chemical reaction network (CRN) theory; see especially \cite{HornJackson, Horn, Feinberg2}. The assumption of existence of a complex-balanced equilibrium has powerful consequences for the dynamical behavior, precluding multi-stability and oscillations. In this note we will revisit the notion of complex balancing by a systematic use of notions and results from algebraic graph theory, in particular the Laplacian matrix and Kirchhoff's Matrix Tree theorem (see also \cite{mirzaev, gunawardena} for other uses of this theorem in chemical reaction dynamics). This will result in a constructive necessary and sufficient condition for complex balancing. Furthermore, motivated by recent work in \cite{dickinson} expanding on \cite{FeinbergWeg}, we will provide a new perspective and results on the Wegscheider conditions for detailed balancing and formal balancing as introduced in \cite{dickinson}.

The structure of this note is as follows. Section 2 gives a brief recap of the basic framework of CRN theory from an algebraic graph theory perspective, based on \cite{vds,rao,vdsIJC}. Section 3 introduces Kirchhoff's Matrix Tree theorem and shows how the application of this theorem leads to an improved, and more directly verifiable, condition for complex balancing as compared to \cite{Horn,Feinberg2}. Section 4 relates Kirchhoff's Matrix Tree theorem to the notion of formal balancing, cf. \cite{dickinson} and \cite{FeinbergWeg}, and shows how this leads to an insightful graph-theoretic proof of the result obtained in \cite{dickinson} that complex balancing together with formal balancing implies detailed balancing and conversely.

\smallskip

\noindent\emph{\bf Notation}:  
The space of $n$-dimensional real vectors consisting of all strictly positive entries is denoted by $\mR_+^{n}$.
%
The mapping
$\Ln : \mathbb{R}_+^n \to \mathbb{R}^n, \quad x \mapsto \Ln x,$ is the elementwise logarithm, and is defined as the mapping whose $i$-th component is given by
$\ln(x_i).$
Similarly, $\Exp : \mathbb{R}^n \to \mathbb{R}_+^n$ is the mapping whose $i$-th component is given by $\exp x_i$. Furthermore, for two vectors $x,y \in \mathbb{R}_+^n$ we let $\frac{x}{y}$ denote the vector in $\mathbb{R}_+^n$ with $i$-th component $\frac{x_i}{y_i}$.
Finally, $\mathds{1}_n$ denotes the $n$-dimensional vector with all entries equal to $1$.

\smallskip
\noindent\emph{\bf Some graph-theoretic notions} \cite{Bollobas}: A directed graph\footnote{Sometimes called a {\it multigraph} since we allow for multiple edges between vertices.} $\mathcal{G}$ with $c$ vertices and $r$ edges is characterized by its $c \times r$ {\it incidence matrix}, denoted by $D$. Each column of $D$ corresponds to an edge of the graph, and contains exactly one element $1$ at the position of the head vertex of this edge and exactly one $-1$ at the position of its tail vertex; all other elements are zero. Clearly, $\mathds{1}_c^TD=0$. The graph is {\it connected} if any vertex can be reached from any other vertex by following a sequence of edges; direction not taking into account. It holds that $\rank D= c- \ell$, where $\ell$ is the number of {\it connected components} of the graph. In particular, $\mathcal{G}$ is connected if and only if $\ker D^T = \spa \mathds{1}_c$. The graph is {\it strongly connected} if any vertex can be reached from any other vertex, following a sequence of {\it directed} edges. A subgraph of $\mathcal{G}$ is a directed graph whose vertex and edge set are subsets of the vertex and edge set of $\mathcal{G}$. A graph is {\it acyclic} (or, does not contain cycles) if and only if $\ker D =0$.
A {\it spanning tree} of a directed graph $\mathcal{G}$ is a connected, acyclic subgraph of $\mathcal{G}$ that contains all vertices of $\mathcal{G}$.

\section{Recall of complex-balanced chemical reaction networks}
In this section, in order to set the stage, we will briefly recall the well-established framework of (isothermal) chemical reaction network (CRN) theory, originating in the work of Horn, Jackson and Feinberg in the 1970s \cite{HornJackson, Feinberg2, Horn, FeinbergHorn1974}.
Consider a chemical reaction network with $m$ chemical species (metabolites) with concentrations $x \in \mathbb{R}^m_+$, among which $r$ chemical reactions take place. The left-hand sides of the chemical reactions are called {\it substrate complexes} and the right-hand sides the {\it product complexes}. 
To each chemical complex (substrate and/or product) of the reaction network one can associate a {\it vertex} of a graph, and to every reaction (from substrate to product complex) a directed edge (with tail vertex the substrate and head vertex the product complex). 
Let $c$ be the total number of complexes involved in the chemical reaction network, then the resulting directed graph $\mathcal{G}$ with $c$ vertices and $r$ edges is called the {\it graph of complexes}\footnote{In the literature sometimes also referred to as {\it reaction graphs}.}, and is defined by its $c \times r$ incidence matrix $D$.
Furthermore we define the $m \times c$ {\it complex composition matrix}\footnote{In \cite{Horn,HornJackson}, it is called {\it complex matrix} and is denoted by $Y$.} $Z$ with non-negative integer elements expressing the composition of the complexes in terms of the chemical species: its $k$-th column expresses the composition of the $k$-th complex. The dynamics of the chemical reaction network takes the well-known form
\begin{equation}
\dot{x} = Sv(x)=ZDv(x), 
\end{equation}
where $v(x)$ is the vector of {\it reaction rates}, and $S=ZD$ the {\it stoichiometric matrix}. 

The most basic way to define $v(x)$ is {\it mass action kinetics}. For example, the mass action kinetics reaction rate of the reaction $X_1 + 2X_2 \to X_3$  is given as $v(x) = kx_1x_2^2$ with $k >0$ a reaction constant. In general, for a single reaction with substrate complex $\mathcal{S}$ specified by its corresponding column $Z_{\mathcal{S}} = \begin{bmatrix} 
Z_{\mathcal{S}1} & \cdots Z_{\mathcal{S}m} \end{bmatrix}^T$ of the complex composition matrix $Z$, the mass action kinetics reaction rate is given by
\[
kx_1^{Z_{\mathcal{S}1}}x_2^{Z_{\mathcal{S}2}} \cdots x_m^{Z_{\mathcal{S}m}},
\]
which can be rewritten as $k \exp (Z_{\mathcal{S}}^T \Ln x), x \in \mathbb{R}^m_+$. 
Hence the reaction rates of the total reaction network are given by 
\[
v_j(x) = k_j \exp (Z_{\mathcal{S}_{j}}^T \Ln x), \quad j=1, \cdots,r,
\]
where $\mathcal{S}_{j}$ is the substrate complex of the $j$-th reaction with reaction constant $k_j >0$. 
This yields the following compact description of the rate vector $v(x)$. Define the $r \times c$ matrix $K$ as the matrix whose $(j,\sigma)$-th element equals $k_j$
if the $\sigma$-th complex is the substrate complex for the $j$-th reaction, and zero otherwise. Then 
$
v(x) = K \Exp (Z^T \Ln x), \,x \in \mathbb{R}^m_+,$
and the dynamics of the mass action reaction network takes the form
\begin{equation}\label{closed}
\dot{x} = ZDK\Exp (Z^T \Ln x), \quad x \in \mathbb{R}^m_+
\end{equation}
It can be easily verified that the $c \times c$ matrix $L:= - DK$ has nonnegative diagonal elements and nonpositive off-diagonal elements. Moreover, since $\mathds{1}_c^TD=0$ also $\mathds{1}_c^TL=0$, i.e., the column sums of $L$ are all zero. Hence $L$ defines a weighted {\it Laplacian matrix}\footnote{In \cite{HornJackson, Horn} (minus) this matrix is called the {\it kinetic matrix}, and is denoted by $A$.} for the graph of complexes $\mathcal{G}$.

The aim of CRN theory, starting with \cite{Feinberg2,Horn,HornJackson}, is to analyze the dynamical properties of (\ref{closed}), and in particular to derive conditions which ensure a dynamical behavior which is independent of the precise values of the reaction constants (which are often poorly known or varying). This has culminated in the {\it deficiency-zero} and {\it deficiency-one} theorems (see e.g. \cite{Feinberg1}), while a somewhat complementary approach is based on the assumption of existence of a {\it complex-balanced equilibrium} \cite{Feinberg2, Horn}, generalizing the classical notion of a detailed-balanced equilibrium.
\begin{definition}
A chemical reaction network (\ref{closed}) is called {\it complex-balanced}\footnote{Note that in older references, e.g. \cite{Horn, Feinberg2}, a reaction network (or {\it mechanism}) is called complex-balanced if there exists a complex-balanced equilibrium for {\it all} {\red positive reaction constants}.} if there exists an equilibrium $x^* \in \mathbb{R}_+^m$, called a {\it complex-balanced equilibrium}, satisfying
\begin{equation}\label{bal}
Dv(x^*) = - L\Exp (Z^T \Ln x^*) =0
\end{equation}
\end{definition}
Chemically (\ref{bal}) means that at the complex-balanced equilibrium $x^*$ not only all the chemical species, but also the complexes remain constant; i.e., for each complex the total inflow (from the other complexes) equals the total outflow (to the other complexes). 

The assumption of complex balancing has been shown to have strong implications for the dynamical properties of (\ref{closed}); see in particular the classical papers \cite{HornJackson, Horn, Feinberg2}. As detailed in \cite{rao}, expanding on \cite{vds}, these properties can be easily proved by defining the diagonal matrix
\begin{equation}\label{K}
\Xi(x^*) :=  \diag \big(\exp (Z_i^T \Ln x^*)\big)_{i=1, \cdots, c},
\end{equation}
and rewriting the dynamics (\ref{closed}) into the form
\begin{equation}\label{masterequation}
\dot{x} = - Z\mathcal{L}(x^*) \Exp (Z^T \Ln (\frac{x}{x^*})),\quad 
\mathcal{L}(x^*) := L \Xi(x^*).
\end{equation}
The key point is that since $\mathcal{L}(x^*) \Exp (Z^T \Ln \left(\frac{x}{x^*}\right)) = 0$ for $x=x^*$, and $\Exp (Z^T \Ln (\frac{x^*}{x^*})) =  \mathds{1}_c$, the transformed matrix $\mathcal{L}(x^*)$ satisfies
\begin{equation}
\mathcal{L}(x^*) \mathds{1}_c =0, \quad \mathds{1}_c^T\mathcal{L}(x^*) =0,
\end{equation}
and thus is a {\it balanced} Laplacian matrix (column sums {\it and} row sums are zero). 
Together with {\it convexity} of the exponential function this implies the following key fact.
\begin{proposition}\label{fundamental}
$\gamma^T \mathcal{L}(x^*) \Exp (\gamma) \geq 0$ for any $\gamma \in \mathbb{R}^c$, with equality if and only if $D^T \gamma =0$.
\end{proposition}
\smallskip
\noindent
First property which directly follows \cite{rao} from Proposition \ref{fundamental} is the classical result \cite{HornJackson, Horn, Feinberg2} that all positive equilibria are in fact complex-balanced, and that given one complex-balanced equilibrium $x^*$ the set of {\it all} positive equilibria is given by
\begin{equation}\label{equilibria}
\mathcal{E} := \{ x^{**} \in \mathbb{R}^m_+ \mid S^T \Ln x^{**}  = S^T \Ln x^{*} \}
\end{equation}
Furthermore, using an elegant result from \cite{Feinberg1}, there exists for every initial condition $x_0 \in \mR^m_+$ a {\it unique} $x^{* *}\in \mathcal{E}$ such that $x^{**} - x_0 \in \im S$.  By using the Lyapunov function
\begin{equation}\label{Gibbs}
G(x) =x^T \mathrm{Ln}\left(\frac{x}{x^{**}}\right) + \left(x^{**} - x \right)^T \mathds{1}_m
\end{equation}
Proposition \ref{fundamental} then implies that the vector of concentrations $x(t)$ starting from $x_0$ will converge to $x^{**}$; at least if the reaction network is persistent\footnote{The reaction network is called {\it persistent} if for every $x_0 \in \mathbb{R}_+^{m}$ the $\omega$-limit set of the dynamics (\ref{closed}) does not intersect the boundary of $\bar{\mathbb{R}}_+^{m}$. It is generally believed that most reaction networks are persistent. However, up to now this {\it persistence conjecture} has been only proved in special cases (cf. \cite{Anderson}, \cite{Siegel}, \cite{Angeli2011} and the references quoted in there).}.
The chemical interpretation is that $G$ is (up to a constant) the {\it Gibbs' free energy} with gradient vector $\frac{\partial G}{\partial x}(x) = \mathrm{Ln}\left(\frac{x}{x^{**}}\right)$ being the {\it chemical potentials}. See e.g. \cite{opk,vds,Schaft2013Lyon} for further information\footnote{The form (\ref{masterequation}) also provides a useful starting point for {\it structure-preserving model reduction} \cite{vds, rao, RAO2014}.}.

\section{A graph-theoretic characterization of complex-balancing}
In this section we will expand on earlier investigations to characterize the existence of a complex-balanced equilibrium (see in particular \cite{Horn}, \cite{Feinberg2}, \cite{dickinson}, and the references quoted therein), and derive a new necessary and sufficient condition for complex balancing which can be constructively verified.

First note that by the definition of $\Ln : \mathbb{R}_+^m \to \mathbb{R}^m$ the  existence of an $x^* \in \mathbb{R}^m_+$ such that $L \Exp (Z^T\Ln x^*) =0$ (i.e., $x^*$ is a complex-balanced equilibrium) is equivalent to the existence of a vector $\mu^* \in \mathbb{R}^m$ such that
\begin{equation}
L \Exp (Z^T\mu^*)=0,
\end{equation}
or equivalently, $\Exp(Z^T\mu^*) \in \ker L$. Furthermore, note that $\Exp(Z^T\mu^*) \in \mathbb{R}^c_+$.

In case the graph $\mathcal{G}$ is connected the kernel of $L$ is $1$-dimensional, and a vector $\rho \in \bar{\mathbb{R}}^c_+$ (the closure of the positive orthant) with $\rho \in \ker L$ can be computed by what is sometimes called {\it Kirchhoff's Matrix Tree theorem}\footnote{This theorem goes back to the classical work of Kirchhoff on resistive electrical circuits \cite{Kirchhoff}; see \cite{Bollobas} for a succinct treatment. Nice accounts of the Matrix Tree theorem in the context of chemical reaction networks can be found in \cite{mirzaev, gunawardena}. In \cite{dickinson} Kirchhoff's Matrix Tree theorem is mentioned and exploited in the closely related, but different, context of investigating how far complex balancing is from detailed balancing; see the next section.}, which for our purposes can be summarized as follows. 
Denote the $(i,j)$-th cofactor of $L$ by $C_{ij}=(-1)^{i+j}M_{i,j}$, where $M_{i,j}$ is the determinant of the $(i,j)$-th minor of $L$, which is the matrix obtained from $L$ by deleting its $i$-th row and $j$-th column. Define the adjoint matrix $\mathrm{adj}(L)$ as the matrix with $(i,j)$-th element given by $C_{ji}$. It is well-known that
\begin{equation}\label{adjoint}
L \cdot \mathrm{adj}(L) = (\det{L})I_c =0
\end{equation}
Furthermore, since $\mathds{1}_c^TL=0$ the sum of the rows of $L$ is zero, and hence by the properties of the determinant function it directly follows that $C_{ij}$ does not depend on $i$; implying that $C_{ij} = \rho_j, \, j=1, \cdots, c$. Therefore by defining $\rho := (\rho_1, \cdots, \rho_c)^T$, it follows from (\ref{adjoint}) that $L\rho=0$. Furthermore, cf. \cite[Theorem 14 on p.58]{Bollobas}, $\rho_i$ is equal to {\it the sum of the products of weights of all the spanning trees of} $\mathcal{G}$ {\it directed towards} vertex $i$. In particular, it follows that $\rho_j \geq 0, j=1, \cdots,c$. In fact, $\rho \neq 0$ if and only if $\mathcal{G}$ has a spanning tree. Furthermore, since for every vertex $i$ there exists at least one spanning tree directed towards $i$ if and only if the graph is strongly connected, we may conclude that $\rho \in \mathbb{R}^c_+$ if and only if the graph is {\it strongly connected}. 
\begin{example}\label{excyclic}
Consider the cyclic reaction network
{\red 
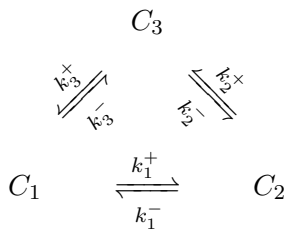
\begin{center}
\begin{tabular}{c c c}
& $C_3$ & \\
& {\rotatebox[origin=c]{45}{$\xleftrightharpoons[k_3^-]{\ k_3^+ \ }$}} \ {\rotatebox[origin=c]{-45}{$\xleftrightharpoons[k_2^-]{\ k_2^+ \ }$}} & \\
& $C_1$ \ \ \ \ \ \ $\xrightleftharpoons[k_1^-]{\ k_1^+ \ }$  \ \ \ \ \ \ $C_2$ &
\end{tabular}
\end{center}}
\vspace{0.3cm}

\noindent in three (unspecified) complexes $C_1, C_2, C_3$. The Laplacian matrix is given as
\[
L = \begin{bmatrix} k_1^+ + k_3^- & - k_1^- & - k_3^+ \\
- k_1^+ & k_1^- + k_2^+ & -k_2^- \\
-k_3^- & - k_2^+ & k_3^+ + k_2^-
\end{bmatrix}
\]
By Kirchhoff's Matrix Tree theorem the corresponding vector $\rho$ satisfying $L \rho=0$ is given as
\[
\rho = \begin{bmatrix} k_2^+k_3^+ + k_1^-k_3^+ + k_1^-k_2^- \\
k_1^+k_3^+ + k_1^+k_2^- + k_2^-k_3^- \\
k_1^+k_2^+ + k_2^+k_3^- + k_1^-k_3^- 
\end{bmatrix},
\]
where each term corresponds to one of the three weighted spanning trees pointed towards the three vertices.

\end{example}
In case the graph $\mathcal{G}$ is not connected the same analysis can be performed on any of its connected components. 
\begin{remark}
The {\it existence} (not the explicit {\it construction}) of $\rho \in \mathbb{R}^c_+$ satisfying $L \rho=0$ already follows from the Perron-Frobenius theorem \cite{Horn}, \cite[Lemma V.2]{sontag}; exploiting the fact that the off-diagonal elements of $-L:=DK$ are all nonnegative\footnote{This implies that there exists a real number $\alpha$ such that $-L + \alpha I_m$ is a matrix with all elements nonnegative. Since the set of eigenvectors of $-L$ and $-L + \alpha I_m$ are the same, and moreover by $\mathds{1}^TL=0$ there cannot exist a positive eigenvector of $-L$ corresponding to a non-zero eigenvalue, the application of Perron-Frobenius to $-L + \alpha I_m$ yields the result; see \cite[Lemma V.2]{sontag} for details.}.
\end{remark}
Returning to the existence of $\mu^* \in \mathbb{R}^m$ satisfying $L \Exp (Z^T\mu^*)=0$ this implies the following.
Let $\mathcal{G}_j, \, j=1, \cdots, \ell,$ be the connected components of the graph of complexes $\mathcal{G}$. For each connected component, define the vectors $\rho^1, \cdots, \rho^{\ell}$  as above by Kirchhoff's Matrix Tree theorem (i.e., as cofactors of the corresponding diagonal sub-blocks of $L$ or as sums of products of weights along spanning trees). Define the total vector $\rho$ as the stacked column vector $\rho := \mathrm{col} ( \rho^1, \cdots, \rho^{\ell})$. 
Partition correspondingly the composition matrix $Z$ as $Z= [Z_1 \cdots Z_{\ell}]$.
Then there exists $\mu^* \in \mathbb{R}^m$ satisfying $L \Exp (Z^T\mu^*)=0$ if and only if each connected component $\mathcal{G}_j$, $j=1,\cdots,\ell,$ is strongly connected and 
\begin{equation}
\Exp (Z_j^T\mu^*) = \beta_j \rho^j, \qquad \beta_j>0.
\end{equation}
This in its turn is equivalent to strong connectedness of each connected component $\mathcal{G}_j$ and the existence of constants $\beta'_j$ such that
$
Z_j^T\mu^* = \Ln \rho^j + \beta_j' \mathds{1}, j=1, \cdots, \ell
$.
Furthermore, this is equivalent to strong connectedness of each connected component of $\mathcal{G}$, and
\begin{equation}\label{final}
\Ln \rho \in \im Z^T + \ker D^T
\end{equation}
Finally, (\ref{final}) is equivalent to
\begin{equation}\label{final1}
D^T\Ln \rho \in \im D^TZ^T = \im S^T
\end{equation}
Summarizing we have obtained
\begin{theorem}\label{Kirchhoff}
The reaction network dynamics $\dot{x} = - ZL\Exp (Z^T \Ln x)$ on the graph of complexes $\mathcal{G}$ is complex-balanced if and only if each connected component of $\mathcal{G}$ is strongly connected (or, equivalently, $\rho \in \mathbb{R}^c_+$) and (\ref{final1}) is satisfied, where the elements of the sub-vectors $\rho^j$ of $\rho$ are obtained by Kirchhoff's Matrix Tree theorem applied to $L$ for each $j$-th connected component of $\mathcal{G}$. Furthermore, for a complex-balanced reaction network a balanced Laplacian matrix\footnote{It can be easily seen that the balanced Laplacian matrices $\mathcal{L}(x^*)$ for different equilibria $x^*$ just differ from each other by a positive multiplicative constant for each connected component of $\mathcal{G}$.} $\mathcal{L}(x^*)$ defined in (\ref{masterequation}) is given as
\begin{equation}
\mathcal{L}(x^*) = L \diag (\rho_1, \cdots, \rho_c)
\end{equation}
\end{theorem}
\begin{remark}
The above theorem is a restatement of Theorem 3C in \cite{Horn}; the main difference being that in \cite{Horn} the positive vector $\rho \in \ker L$ remains unspecified, while in our case it is explicitly given by Kirchhoff's Matrix Tree theorem.
\end{remark}
We directly obtain the following corollary stated before in \cite[Eq. (3.21)]{Horn}:
\begin{corollary}
The reaction network dynamics $\dot{x} = - ZL\Exp (Z^T \Ln x)$ is complex-balanced if and only if $\rho \in \mathbb{R}^c_+$ and 
\begin{equation}\label{complexWeg}
\rho_1^{\sigma_1} \cdot \rho_2^{\sigma_2} \cdots \cdot \rho_c^{\sigma_c} =1,
\end{equation}
for all vectors $\sigma =(\sigma_1, \sigma_2, \cdots, \sigma_c)^T \in \ker Z \cap \im D $. In particular, if $\ker Z \cap \im D =0$ ({\it zero-deficiency} \cite{Horn, FeinbergHorn1974, Feinberg1}) then $\dot{x} = - ZL\Exp (Z^T \Ln x)$ is complex-balanced.
\end{corollary}
\begin{proof}
$\Ln \rho \in \im Z^T + \ker D^T$ if and only if $\sigma^T \Ln \rho=0$ for all $\sigma \in (\im Z^T + \ker D^T)^{\perp} = \ker Z \cap \im D$, or equivalently
\[
0=\sigma_1 \ln \rho_1 + \cdots + \sigma_c \ln \rho_c = \ln \rho_1^{\sigma_1} + \cdots +  \ln \rho_c^{\sigma_c} =
\ln (\rho_1^{\sigma_1} \cdots \rho_c^{\sigma_c})
\]
for all $\sigma \in \ker Z \cap \im D$. 
\end{proof}
\begin{example}\label{excyclic1}
Consider Example \ref{excyclic} for the special case $k_1^- = k_2^- = k_3^-=0$ (irreversible reactions). Then the vector $\rho$ reduces to
\[
\rho = \begin{bmatrix} 
k_2^+k_3^+  \\
k_1^+k_3^+  \\
k_1^+k_2^+ 
\end{bmatrix}
\]
The reaction network with complex composition matrix $Z$ is complex-balanced if and only $k_i^+>0, i=1,2,3,$ and
\[
D^T \Ln \rho \in \im D^TZ^T, \quad D= \begin{bmatrix} -1 & 0 & 1 \\ 1 & -1 & 0 \\ 0 & 1 & -1 \end{bmatrix},
\]
This last condition can be further written out as
\[
\begin{bmatrix}
\ln \frac{k_1^+}{k_2^+} \\
\ln \frac{k_2^+}{k_3^+} \\
\ln \frac{k_3^+}{k_1^+} 
\end{bmatrix}
\in \im D^TZ^T
\]
As a mathematical example, take the complex composition matrix $Z= \begin{bmatrix} 1 & 0 & 2 \\ 1 & 1 & 1 \end{bmatrix}$ (corresponding to the complexes $X_1 + X_2, X_2, 2X_1 + X_2$). In this case the network is complex-balanced if and only if $k_1^+>0, k_2^+>0, k_3^+>0,$ and $(k_1^+)^2 = k_2^+ k_3^+$.
\end{example}

\section{Relation with the Wegscheider conditions and detailed balancing}
In this section {\red we} will relate the conditions for complex balancing as obtained in the previous section to 'Wegscheider-type conditions'. This will also relate complex balancing to the classical concept of 'detailed balancing'.

Throughout this section we will consider {\it reversible} chemical reaction networks, in which case the edges of $\mathcal{G}$ come in pairs: if there is a directed edge from vertex $i$ to $j$ then there also is a directed edge from $j$ to $i$ (and in the case of multiple edges from $i$ to $j$ there are as many edges from $i$ to $j$ as edges from $j$ to $i$). This means that the connected components of $\mathcal{G}$ are always strongly connected, or equivalently, that $\rho$ as obtained from Kirchhoff's Matrix Tree theorem is in $\mathbb{R}^c_+$.

Define the {\it undirected} graph $\bar{\mathcal{G}}$ as having the same vertices as $\mathcal{G}$ but half its number of edges, by replacing every pair of oppositely directed edges of $\mathcal{G}$ by one undirected edge of $\bar{\mathcal{G}}$. Denote the number of edges of $\bar{\mathcal{G}}$ by $\bar{r} = \frac{1}{2}r$. Endow subsequently $\bar{\mathcal{G}}$ with an arbitrary orientation (all results in the sequel will be independent of this orientation), and denote the resulting incidence matrix by $\bar{D}$. Clearly, after possible reordering of the edges, $\bar{D}$ is related to the incidence matrix $D$ of $\mathcal{G}$ as
\begin{equation}\label{DD}
D = \begin{bmatrix} \bar{D} & - \bar{D} \end{bmatrix}
\end{equation}
To the $j$-th edge of $\bar{\mathcal{G}}$ there now correspond two reaction constants $k_{j}^+, k_{j}^-$ (the forward and reverse reaction constants with respect to the chosen orientation of $\bar{\mathcal{G}}$).
Then define the {\it equilibrium constants} $K^{\mathrm{eq}}_j := \frac{k_{j}^+}{k_{j}^-}, j=1, \cdots, \bar{r}$, and the vector $K^{\mathrm{eq}} := (K^{\mathrm{eq}}_1, \cdots, K^{\mathrm{eq}}_{\bar{r}})^T$.

Recall \cite{FeinbergWeg, Schuster} that the reaction network is called {\it detailed-balanced}\footnote{This means, see e.g. \cite{opk,HornJackson,vds}, that there exists an equilibrium for which every forward reaction is balanced by its reverse reaction.} if and only if it satisfies
\begin{equation}\label{detbal}
\Ln K^{\mathrm{eq}} \in \im \bar{S}^T,
\end{equation}
where $\bar{S}:=Z\bar{D}$ is the stoichiometric matrix of the reversible network with graph $\bar{\mathcal{G}}$. This is equivalent to
\[
\sigma_1 {\red \ln} K^{\mathrm{eq}}_1 + \cdots + \sigma_{\bar{r}} {\red \ln} K^{\mathrm{eq}}_{\bar{r}} =0
\]
for all $\sigma=(\sigma_1, \cdots, \sigma_{\bar{r}})$ such that $\sigma^T \bar{S}^T=0$. Writing out 
$K^{\mathrm{eq}}_j = \frac{k_{j}^+}{k_{j}^-}$ this is seen to be equivalent to
\begin{equation}\label{Weg}
(k_1^+)^{\sigma_1} \cdots (k_{\bar{r}}^+)^{\sigma_{\bar{r}}} = (k_1^-)^{\sigma_1} \cdots (k_{\bar{r}}^-)^{\sigma_{\bar{r}}}
\end{equation}
for all $\sigma$ such that $\bar{S}\sigma=0$, known as the (generalized) {\it Wegscheider conditions}.

Recently in \cite{dickinson} the notion of {\it formally balanced} was introduced, based on Feinberg's circuit conditions in \cite{FeinbergWeg}, and weakening the above Wegscheider conditions. In our set-up this notion is defined as follows.
\begin{definition}
The reversible reaction network $\bar{\mathcal{G}}$ with incidence matrix $\bar{D}$, and vector of equilibrium constants $K^{\mathrm{eq}}$ is called {\it formally balanced} if 
\[
\Ln K^{\mathrm{eq}} \in \im \bar{D}^T
\]
\end{definition}
Since $\bar{S}=Z \bar{D}$ 'formally balanced' is trivially implied by 'detailed-balanced', while if $\im \bar{S}^{\red T} = \im \bar{D}^{T}$ ({\it zero-deficiency}) the reverse holds. Furthermore, any reaction network with acyclic $\bar{\mathcal{G}}$ (and thus $\ker \bar{D} =0$) is automatically formally balanced. 

As above, the notion of 'formally balanced' is seen to be equivalent to
\[
\sigma_1 {\red \ln}K^{\mathrm{eq}}_1 + \cdots + \sigma_r {\red \ln}K^{\mathrm{eq}}_{\bar{r}} =0
\]
for all $\sigma=(\sigma_1, \cdots, \sigma_{\bar{r}})^T$ such that $\sigma^T \bar{D}^T=0$, which in turn is equivalent {\red to} 
\begin{equation}\label{weakWeg}
(k_1^+)^{\sigma_1} \cdots (k_{\bar{r}}^+)^{\sigma_{\bar{r}}} = (k_1^-)^{\sigma_1} \cdots (k_{\bar{r}}^-)^{\sigma_{\bar{r}}}
\end{equation}
for all $\sigma$ such that $\bar{D} \sigma=0$ (that is, for all cycles $\sigma$). We will refer to (\ref{weakWeg}) as the {\it weak Wegscheider conditions}. Note that the weak Wegscheider conditions only depend on the structure of the graph $\bar{\mathcal{G}}$ (i.e., its cycles) and the equilibrium constants, and {\it not} on the complex composition matrix $Z$ as in the case of the 'strong' Wegscheider conditions (\ref{Weg}). 

\begin{theorem}\label{formal}
Consider a reversible chemical reaction network given by the graph $\bar{\mathcal{G}}$ with incidence matrix $\bar{D}$, and with $\rho$ determined by $L$. The following statements are equivalent
\begin{enumerate}
\item
$L \diag(\rho_1, \cdots, \rho_c)$ is symmetric
\item
\smallskip
$
\Ln (K^{\mathrm{eq}}) = \bar{D}^T \Ln \rho
$
\item
\smallskip
$
\Ln (K^{\mathrm{eq}}) \in \im \bar{D}^T$ \quad (formally balanced)
\end{enumerate}
\end{theorem}
\begin{proof}
$(1) \Leftrightarrow (2)$\\
Let us first prove the equivalence between ($1$) and ($2$). Consider the $i$-th and $j$-th vertex of $\bar{\mathcal{G}}$, and suppose that the orientation has been taken such that the $\alpha$-th edge between $i$ and $j$ is such that $i$ is the tail vertex and $j$ is the head vertex. Then the $(i,j)$-th element of $L \diag(\rho_1, \cdots, \rho_c)$ is given by $k_{\alpha}^- \rho_j$, while the $(j,i)$-th element equals $k_{\alpha}^+ \rho_i$. Symmetry of $L \diag(\rho_1, \cdots, \rho_c)$ thus amounts to
\[
k_{\alpha}^- \rho_j = k_{\alpha}^+ \rho_i
\]
for all pairs of vertices $i,j$.
On the other hand, 
the $\alpha$-th element of the vector $\bar{D}^T \Ln \rho$ is given by
\[
\ln \rho_j - \ln \rho_i
\]
while the $\alpha$-th element of $\Ln (K^{\mathrm{eq}})$ is given by
\[
\ln \frac{k_{\alpha}^+}{k_{\alpha}^-} = \ln k_{\alpha}^+ - \ln k_{\alpha}^-
\]
Equality of $\ln \rho_j - \ln \rho_i$ and $\ln k_{\alpha}^+ - \ln k_{\alpha}^-$ is thus equivalent to
\[
 \ln  \rho_j + \ln k_{\alpha}^- = \ln  \rho_i + \ln k_{\alpha}^+
\]
which in its turn is equivalent to $k_{\alpha}^+ \rho_i = k_{\alpha}^- \rho_j$ as above.\\
$(2) \Leftrightarrow (3)$\\
Obviously ${\red (2)}$ implies ${\red (3)}$. For the reverse implication, consider any pair of vertices linked by an edge of the graph $\bar{\mathcal{G}}$. Depending on the orientation of $\bar{\mathcal{G}}$ refer to one vertex as the tail vertex $t$ and the other vertex as the head vertex $h$. Refer to the positive reaction constant from $t$ to $h$ by $k^+$ and to the negative reaction constant by $k^-$. Now consider a spanning tree directed towards $t$ with product of weights denoted by $\tau_t$. In case the edge between $t$ and $h$ is part of this spanning tree then it follows that by reversing the orientation of this edge it defines a spanning tree directed towards to $h$ with product of weights denoted by $\tau_h$. It follows directly that
\begin{equation}\label{comb}
\frac{k^+}{k^-} \tau_t= \tau_h
\end{equation}
In case the edge between $t$ and $h$ is {\it not} part of this spanning tree then divide the edges of the spanning tree into two sets; the set $E_1$ containing the edges of the part of the spanning tree from $t$ to $h$ (containing say $\ell$ edges) and the set $E_2$ containing the remaining edges of the spanning tree. Observe that $E_1$ {\it together} with the edge from $t$ to $h$ forms a cycle of the graph $\bar{\mathcal{G}}$. Since 
$\Ln (K^{\mathrm{eq}}) \in \im \bar{D}^T$ it follows that for the reaction constants along this cycle (choosing an appropriate orientation),
\begin{equation}\label{comb1}
k^+ \cdot k_1^+ \cdots k_{\ell}^+ = k^- \cdot k_1^- \cdots k_{\ell}^-.
\end{equation}
Now within the spanning tree directed towards $t$, if the orientation of each of the edges of $E_1$ is reversed, we obtain another spanning tree directed towards $h$ with product of weights denoted again by $\tau_h$. By using (\ref{comb1}) it is readily verified that also in this case we obtain the same relation (\ref{comb}). Summing up over all spanning trees we thus obtain the equality
\begin{equation}\label{comb2}
\frac{k^+}{k^-} {\red \rho}_t= {\red \rho}_h,
\end{equation}
which can be equivalently written as
$\ln \frac{k^+}{k^-} = \ln {\red \rho}_h - \ln {\red \rho}_t$. Doing this for all adjacent vertices $t$ and $h$ this exactly amounts to the required equality $\Ln (K^{\mathrm{eq}}) = \bar{D}^T \Ln \rho$.
\end{proof}

\begin{example}
Consider again the reaction network described in Example \ref{excyclic} (without specifying the complexes $C_1, C_2, C_3$). The transformed Laplacian matrix is computed as
\[
\begin{array}{rcl}
\mathcal{L} = \left[\begin{smallmatrix} k_1^+ + k_3^- & - k_1^- & - k_3^+ \\
- k_1^+ & k_1^- + k_2^+ & -k_2^- \\
-k_3^- & - k_2^+ & k_2^- + k_3^+
\end{smallmatrix}\right]\left[
\begin{smallmatrix} k_2^+k_3^+ + k_1^-k_3^+ + k_1^-k_2^- & 0 & 0 \\
0 & k_1^+k_3^+ + k_1^+k_2^- + k_2^-k_3^- & 0 \\
0 & 0 & k_1^+k_2^+ + k_2^+k_3^- + k_1^-k_3^- 
\end{smallmatrix}\right] = \\[10mm]
\left[\begin{smallmatrix}
(k_1^+ +k_2^+)(k_2^+k_3^+ + k_1^-k_3^+ + k_1^-k_2^-) & -k_1^-(k_1^+k_3^+ + k_1^+k_2^- + k_2^-k_3^-) & -k_3^+(k_1^+k_2^+ + k_2^+k_3^- + k_1^-k_3^-) \\
-k_1^+(k_2^+k_3^+ + k_1^-k_3^+ + k_1^-k_2^-) & (k_1^- + k_2^+)(k_1^+k_3^+ + k_1^+k_2^- + k_2^-k_3^-) & -k_2^-(k_1^+k_2^+ + k_2^+k_3^- + k_1^-k_3^-) \\
-k_3^-(k_2^+k_3^+ + k_1^-k_3^+ + k_1^-k_2^-) & -k_2^+(k_1^+k_3^+ + k_1^+k_2^- + k_2^-k_3^-) & (k_2^- + k_3^+)(k_1^+k_2^+ + k_2^+k_3^- + k_1^-k_3^-) 
\end{smallmatrix}\right]
\end{array}
\]
which is symmetric if and only if
\begin{equation}\label{formalex}
k_1^+k_2^+k_3^+ = k_1^-k_2^-k_3^-
\end{equation}
On the other hand, $\Ln K^{\mathrm{eq}} \in \bar{D}^T$ amounts to
\[
\begin{bmatrix} \ln \frac{k_1^+}{k_1^-} \\ \ln \frac{k_2^+}{k_2^-} \\ \ln \frac{k_3^+}{k_3^-} \end{bmatrix} \in
\im \begin{bmatrix}-1 & 0 & 1 \\ 1 & -1 & 0 \\ 0 & 1 & -1 \end{bmatrix}
\]
which reduces to $\ln \frac{k_1^+}{k_1^-} + \ln \frac{k_2^+}{k_2^-} + \ln \frac{k_3^+}{k_3^-} =0$, and hence to the same condition (\ref{formalex}). Thus the reaction network is formally balanced if and only if (\ref{formalex}) holds.
\end{example}

\smallskip
Now let us relate all this to the necessary and sufficient conditions for complex balancing obtained before, cf. (\ref{final1}). Note that by (\ref{DD})
\[
D^T \Ln \rho \in \im D^TZ^T \Leftrightarrow \bar{D}^T \Ln \rho \in \im \bar{D}^TZ^T
\]
Hence a reversible reaction network is complex-balanced if and only if $\bar{D}^T \Ln \rho \in \im \bar{D}^TZ^T = \im \bar{S}^T$.

We directly obtain the following corollary proved by other methods in \cite{dickinson}\footnote{The result is formulated in \cite{dickinson} as the equivalence between 'complex-balanced' and 'detailed-balanced' under the assumption of 'formally balanced'.}.
\begin{corollary}\label{cor}
A reversible reaction network is detailed-balanced if and only if it is formally balanced as well as complex-balanced. 
\end{corollary}
\begin{proof}
We have seen before that 'detailed-balanced' implies 'complex-balanced' as well as 'formally balanced'. 
For the converse we note that formally balanced implies that $\Ln (K^{\mathrm{eq}}) \in \im \bar{D}^T$. Hence by Theorem \ref{formal} $\Ln (K^{\mathrm{eq}}) = \bar{D}^T \Ln \rho$. Since furthermore the network is complex-balanced $\bar{D}^T \Ln \rho \in \im \bar{S}^T$. Hence $\Ln (K^{\mathrm{eq}}) = \bar{D}^T \Ln \rho \in \im \bar{S}^T$, i.e., the reaction network is detailed-balanced.
\end{proof}
In case the reversible reaction network is formally balanced the symmetric matrix $\mathcal{L}(x^*)= L \diag (\rho_1, \cdots, \rho_c)$ can be written as
\[
L \diag (\rho_1, \cdots, \rho_c) = \bar{D}\mathcal{K}\bar{D}^T
\]
where $\mathcal{K}$ is the $\bar{r} \times \bar{r}$ diagonal matrix, with $\alpha$-th diagonal element given by $\kappa_{\alpha}:=k_{\alpha}^+ \rho_j = k_{\alpha}^- \rho_i$ where the $\alpha$-th edge of $\bar{\mathcal{G}}$ corresponds to the reversible reaction between the $i$-th and the $j$-th complex. For the interpretation of the positive constants $\kappa_{\alpha}$ as {\it conductances} of the reversible reactions please refer to \cite{Ederer2007, Schaft2013Lyon}.

If additionally the formally balanced reaction network is {\it complex-balanced} (and thus, cf. Corollary \ref{cor}, detailed-balanced), then its dynamics thus takes the form
\begin{equation*}
\dot{x} = - Z \bar{D}\mathcal{K}\bar{D}^T \Exp (Z^T \Ln \left(\frac{x}{x^*}\right))
\end{equation*}
In this case, see e.g. \cite{HornJackson,vds}, all equilibria $x^{**}$ are in fact detailed-balanced equilibria, that is, $\bar{D}^T Z^T \Ln x^{**} = \Ln (K^{\mathrm{eq}}) \, (= \bar{D}^T \Ln \rho )$.

\section{Conclusions}
By a systematic use of notions from algebraic graph theory, in particular the Laplacian matrix and Kirchhoff's Matrix Tree theorem, previously derived results on complex, detailed and formal balancing have been proved in a simple manner. Furthermore, it has resulted in a new necessary and sufficient condition for complex balancing, which can be verified constructively. 

The results obtained in this note can be immediately extended to mass action kinetic reaction networks with constant inflows and mass action outflow exploiting the classical idea of adding a 'zero complex'; see \cite{FeinbergHorn1974} and \cite{vdsIJC} for further details. 

Current research is concerned with the application of the developed framework to questions of occurrence of multi-stability and structure-preserving model reduction; see for the latter also \cite{vds,rao,RAO2014}.

\section*{Compliance with Ethical Standards}
Conflict of Interest: The authors declare that they have no conflict of interest.

\end{document}